\documentclass[12pt]{amsart}

\usepackage{amssymb,amsbsy}
\usepackage{latexsym,array}
\usepackage{verbatim,color}
\usepackage{fullpage,euscript}
\usepackage{xcolor}
\usepackage{tikz}
\usetikzlibrary{arrows,shapes,trees}
\usepackage[colorlinks=true,linkcolor=blue,urlcolor=my_color,citecolor=magenta]{hyperref}

\definecolor{my_color}{rgb}{0,0.5,0.5}
\definecolor{MIXT}{rgb}{0.8,0.5,0.2}
\definecolor{mixt}{rgb}{0.5,0.3,0.2}
\definecolor{sin}{rgb}{0,0.5,0.5}
\definecolor{darkblue}{rgb}{0,0.1,0.8}
\definecolor{redi}{rgb}{0.5,0,0.4}

\numberwithin{equation}{section}

\input {cyracc.def}
\font\tencyr=wncyr10 
\font\tencyi=wncyi10 
\font\tencysc=wncysc10 
\def\rus{\tencyr\cyracc}

\def\rusi{\tencyi\cyracc}
\def\rusc{\tencysc\cyracc}

\newtheorem{thm}{Theorem}[section]
\newtheorem{lm}[thm]{Lemma}
\newtheorem{cl}[thm]{Corollary}

\theoremstyle{remark}
\newtheorem{rmk}[thm]{Remark}

\theoremstyle{definition}
\newtheorem{ex}[thm]{Example} 
\newtheorem{df}{Definition}

\newcommand {\be}{{\mathfrak b}}

\newcommand {\f}{{\mathfrak f}}
\newcommand {\g}{{\mathfrak g}}

\newcommand {\h}{{\mathfrak h}}

\newcommand {\me}{{\mathfrak m}}

\newcommand {\q}{{\mathfrak q}}

\newcommand {\te}{{\mathfrak t}}
\newcommand {\ut}{{\mathfrak u}}

\newcommand {\z}{{\mathfrak z}}

\newcommand {\gln}{{\mathfrak{gl}}_n}

\newcommand{\gt}{\mathfrak}

\newcommand {\eus}{\EuScript}

\newcommand {\gA}{{\eus A}}

\newcommand {\gS}{{\eus S}}
\newcommand {\U}{{\eus U}}
\newcommand {\gZ}{{\eus Z}}
\newcommand{\cam}{\eus{MF}}

\newcommand {\ap}{\alpha}

\newcommand {\vp}{\varphi}

\newcommand {\ca}{{\mathcal A}}

\newcommand {\BV}{{\mathbb V}}

\newcommand {\BZ}{{\mathbb Z}}
\newcommand {\BN}{{\mathbb N}}

\newcommand{\id}{{\mathsf{id}}}

\newcommand {\codim}{{\mathrm{codim\,}}}
\newcommand {\cha}{{\mathrm{char}}}

\newcommand {\ind}{{\mathrm{ind\,}}}
\newcommand {\Lie}{{\mathsf{Lie}}}

\newcommand {\Ima}{{\mathrm{Im\,}}}

\newcommand {\rk}{{\mathsf{rk\,}}}

\newcommand {\trdeg}{{\mathrm{tr.deg\,}}}

\newcommand {\tri}{\mathfrak{sl}_2}

\newcommand {\bb}{{\boldsymbol{b}}}

\newcommand {\PC}{Poisson commutative}

\newcommand {\beq}{\begin{equation}}
\newcommand {\eeq}{\end{equation}}

\renewcommand{\le}{\leqslant}
\renewcommand{\ge}{\geqslant}
\renewcommand{\lg}{\langle}
\newcommand{\rg}{\rangle}

\newcommand {\bbk}{\Bbbk}

\begin{document}
\setlength{\parskip}{3pt plus 2pt minus 0pt}
\hfill { {\color{blue}\scriptsize December 7, 2020}}
\vskip1ex

\title[Reductive subalgebras and Poisson commutativity] 
{Reductive subalgebras of 
semisimple Lie algebras and  Poisson commutativity} 
\author[D.\,Panyushev]{Dmitri I. Panyushev}
\address[D.\,Panyushev]%
{Institute for Information Transmission Problems of the R.A.S., Bolshoi Karetnyi per. 19,
Moscow 127051, Russia}
\email{panyushev@iitp.ru}
\author[O.\,Yakimova]{Oksana S.~Yakimova}
\address[O.\,Yakimova]{Institut f\"ur Mathematik, Friedrich-Schiller-Universit\"at Jena, Jena, 07737, Deutschland}
\email{oksana.yakimova@uni-jena.de}
\thanks{The first author is partially supported by R.F.B.R. grant {\rus N0} 20-01-00515.
The second author is funded by the Deutsche Forschungsgemeinschaft (DFG, German Research Foundation) --- project number 454900253.}
\maketitle
\begin{abstract}
Let $\g$ be a semisimple Lie algebra, $\h\subset\g$ a reductive subalgebra such that 
$\h^\perp$ is a
complementary $\h$-submodule of $\g$. In 1983, Bogoyavlenski claimed that one obtains a Poisson
commutative subalgebra of the symmetric algebra $\gS(\g)$ by taking the subalgebra $\gZ$ generated by
the bi-homogeneous components of all $H\in\gS(\g)^\g$. But this is false, and
we present a counterexample. We also provide a criterion for the Poisson commutativity of such
subalgebras $\gZ$. As a by-product, we prove that $\gZ$ is Poisson commutative if $\h$ is abelian
and describe $\gZ$ in the special case when $\h$ is a Cartan subalgebra. In this case, $\gZ$ appears
to be polynomial and has the maximal transcendence degree $\bb(\g)=\frac{1}{2}(\dim\g+\rk\g)$.
\end{abstract}


\section*{Introduction}

\noindent
\subsection{} The ground field $\bbk$ is algebraically closed and $\cha(\bbk)=0$.
For any finite-dimensional Lie algebra $\q$, the dual space $\q^*$ has a Poisson structure. The algebra
of polynomial functions on $\q^*$, $\bbk[\q^*]$, is isomorphic to the graded symmetric algebra $\gS(\q)$
and the Lie--Poisson bracket $\{\,\,,\,\}$ is defined on the elements of degree one by $\{\xi,\eta\}=[\xi,\eta]$
for $\xi,\eta\in\q$.
There is a method for constructing ``large'' Poisson commutative subalgebras of $\gS(\q)$ that
exploits pairs of {\it compatible Poisson brackets}, see~\cite[Sect.~10]{GZ}, \cite{ukr}.
To apply this, one needs a suitable second Poisson bracket beside $\{\,\,,\,\}=\{\,\,,\,\}_{\q}$.
Let us recall some situations, where this "general method" (=\,method of compatible Poisson brackets) works.

{\bf I.} The  celebrated ``argument shift method'' goes back to~\cite{mf} (if $\q$ is semisimple). It employs an arbitrary $\gamma\in\q^*$ and the Poisson bracket $\{\,\,,\,\}_{\gamma}$, where
$\{x,y\}_{\gamma}=\gamma([x,y])$ for $x,y\in\q$. The brackets $\{\,\,,\,\}$ and $\{\,\,,\,\}_{\gamma}$ are
compatible, and the general method produces the {\it Mishchenko--Fomenko subalgebra\/}
(=\,$\eus{MF}$-{\it subalgebra})
$(\cam)_\gamma\subset \gS(\q)$.
Let $\gS(\q)^\q$ be the {\it Poisson centre\/} of  $(\gS(\q), \{\ ,\ \})$, i.e.,
\[
    \gS(\q)^\q=\{H\in \gS(\q)\mid \{H,x\}=0 \ \ \forall x\in\q\} .
\]
For $F\in\gS(\q)$, let $\partial_{\gamma} F$ be the directional derivative of $F$
with respect to $\gamma\in\q^*$, i.e.,
\[
    \partial_{\gamma}F(x)=\frac{\textsl{d}}{\textsl{d}t} F(x+t\gamma)\Big|_{t=0}.
\]
By the original definition of the $\eus{MF}$-subalgebras~\cite{mf}, $(\cam)_\gamma$ is generated by all
$\partial_\gamma^k F$ with $k\ge 0$ and $F\in\gS(\q)^{\q}$. Since then, the algebras $(\cam)_\gamma$
and their quantum counterparts attracted a great deal of attention, see e.g.~\cite{FFR,m-y,vi90} and
references therein. If $\q$ is reductive and $\gamma$ is regular in $\q^*$, then $(\cam)_\gamma$ is a
maximal Poisson commutative subalgebra  in $\gS(\q)$ of maximal transcendence degree~\cite{mrl}.

{\bf II.} \ Let $\q=\q_0\oplus\q_1$ be a $\BZ_2$-grading, i.e., $[\q_i,\q_j]\subset \q_{i+j\!\pmod 2}$. Then
$\q$ admits the In\"onu--Wigner contraction to the semi-direct product $\tilde\q=\q_0\ltimes\q_1^{\sf ab}$,
and the second bracket is the Lie--Poisson bracket of  $\tilde\q$. (Here $\q$ and $\tilde\q$
are identified as vector spaces.) The compatibility of $\{\,\,,\,\}_{\q}$ and $\{\,\,,\,\}_{\tilde\q}$ stems from
the presence of $\BZ_2$-grading, cf. Section~\ref{subs:contr-compat}. The sum $\q=\q_0\oplus\q_1$
determines the bi-homogeneous decomposition
$\gS(\q)=\bigoplus_{i,j\ge 0} \gS^i(\q_0)\otimes \gS^j(\q_1)$.
Here the general method yields the Poisson commutative subalgebra generated by the bi-homogeneous
components of all $H\in\gS(\q)^\q$. This case has been studied in \cite{m84} and recently in our
article~\cite{OY-alg}. For substantial applications, one has to assume, of course, that $\q$ is semisimple.

\subsection{}      \label{subs:error}
Soon after~\cite{OY-alg} has been accepted,
we came across an article of Bogoyavlenski~\cite{bo}. He
claims that if $\g$ is semisimple, $\f\subset\g$ is reductive and the Killing form of $\g$ is non-degenerate
on $\f$,  then the direct sum $\g=\f\oplus\me$, where $\me=\f^\perp$, allows to  construct  similarly a
Poisson commutative subalgebra of $\gS(\g)$. Namely, a special case of~\cite[Theorem~1]{bo}
(with $n=k=j=1$ in  the original notation) asserts that the bi-homogeneous components of all
$F\in\gS(\g)^\g$ generate a Poisson commutative subalgebra. However,
this is false and we provide a counterexample to that claim.
An explanations for that error is that
here one can also consider the contraction $\tilde\g=\f\ltimes\me^{\sf ab}$  and the Poisson bracket
$\{\,\,,\,\}_{\tilde\g}$ on the vector space $\g\simeq\tilde\g$, but the brackets $\{\,\,,\,\}_{\g}$ and
$\{\,\,,\,\}_{\tilde\g}$ are not necessarily compatible. One can also notice that Bogoyavlenski did not
properly distinguish a Lie algebra and its dual, and his usage of differentials of elements of $\gS(\g)$
is sloppy.

Our main motivation for writing this note was just to clarify and remedy this situation.
However, we also discovered some exciting new phenomena.
Let $\g=\Lie(G)$ be semisimple and $\g=\f\oplus\me$ as above. 
Let $\gZ_{(\g,\f)}$ be the
subalgebra of $\gS(\g)$ generated by the bi-homogeneous components of all $F\in \gS(\g)^\g$. The
results of this note are:
\begin{itemize}
\item[\sf 1)] we provide a criterion for $\gZ_{(\g,\f)}$ to be Poisson commutative;

\item[\sf 2)] using our criterion we prove that $\gZ_{(\mathfrak{sl}_4,\tri)}$ is not Poisson commutative
for the standard embedding $\tri\subset \mathfrak{sl}_4$;

\item[\sf 3)] a corollary of our criterion is that $\gZ_{(\g,\f)}$ is Poisson commutative whenever $\f$ is abelian (e.g. if $\f$ is the Lie algebra of a torus in $G$);

\item[\sf 4)] it is proved that if $\f=\te$ is a Cartan subalgebra of $\g$, then $\gZ_{(\g,\te)}$ is polynomial,
$\trdeg \gZ_{(\g,\te)}=\bb(\g)$, and $\gZ$ is complete on every regular $G$-orbit in $\g$.

\item[\sf 5)] We point out an algebraic extension $\tilde\gZ\supset \gZ_{(\g,\te)}$ such that $\tilde\gZ$
is a {\bf maximal} Poisson commutative subalgebra of $\gS(\g)$  and is still polynomial.
\end{itemize}

\section{Preliminaries on the coadjoint representation}
\label{sect:prelim}

\noindent
Let $Q$\/ be a connected affine algebraic group with $\Lie(Q)=\q$. The symmetric algebra
$\gS(\q)$ over $\bbk$ is identified with the graded algebra of polynomial functions on $\q^*$, and we also
write $\bbk[\q^*]$ for it.
\\ \indent
Let $\q^\xi$ denote the stabiliser in $\q$ of $\xi\in\q^*$. The {\it index of}\/ $\q$, $\ind\q$, is the minimal codimension of $Q$-orbits in $\q^*$. Equivalently,
$\ind\q=\min_{\xi\in\q^*} \dim \q^\xi$. By Rosenlicht's theorem~\cite[I.6]{brion}, one also has
$\ind\q=\trdeg\bbk(\q^*)^Q$. The Lie--Poisson bracket for
$\bbk[\q^*]$ is defined on the elements of degree $1$ (i.e., on $\q$) by $\{x,y\}:=[x,y]$.
Set further $\hat\gamma(x,y)=\gamma([x,y])$ for $\gamma\in\q^*$.
For any $F_1,F_2\in\gS(\q)$ and $\gamma\in\q^*$, we have
\begin{equation} \label{P-gamma}
\{F_1,F_2\}(\gamma)=\hat\gamma(\textsl{d}_\gamma F_1,\textsl{d}_\gamma F_2)
\end{equation}
where $\textsl{d}_\gamma F\in \q$ is the differential of $F\in\gS(\q)$ at $\gamma$.
As $Q$ is connected, we have $\gS(\q)^\q=\gS(\q)^{Q}=\bbk[\q^*]^Q$.
The set of $Q$-{\it regular\/} elements of $\q^*$ is
\beq       \label{eq:regul-set}
    \q^*_{\sf reg}=\{\eta\in\q^*\mid \dim \q^\eta=\ind\q\} .
\eeq
Set $\q^*_{\sf sing}=\q^*\setminus \q^*_{\sf reg}$.
We say that $\q$ has the {\sl codim}--$n$ property if $\codim \q^*_{\sf sing}\ge n$.  By~\cite{ko63}, the semisimple algebras $\g$ have the {\sl codim}--$3$ property.

Set $\bb(\q)=(\dim\q+\ind\q)/2$.
Since the coadjoint orbits are even-dimensional, this number is an integer. If $\q$ is reductive, then
$\ind\q=\rk\q$ and $\bb(\q)$ equals the dimension of a Borel subalgebra. A subalgebra $\gA\subset \gS(\q)$ is said to be {\it Poisson commutative} if $\{\gA,\gA\}=0$.
If $\gA\subset \gS(\q)$ is Poisson commutative, then $\trdeg\gA\le \bb(\q)$, see e.g.~\cite[0.2]{vi90}.

\begin{df} \label{df-compl}
A Poisson commutative subalgebra
 $\gA\subset \gS(\q)$ is said to be {\it complete} on a coadjoint orbit $Q\gamma\subset\q^*$ if
 $\trdeg (\gA\vert_{Q\gamma}) = \frac{1}{2}\dim (Q\gamma)$.
\end{df}
\noindent
The notion of completeness originates from the theory of integrable systems. 

For a subalgebra $A\subset \gS(\q)$ and $\gamma\in\q^*$, set
$\textsl{d}_\gamma A=\left< \textsl{d}_\gamma F \mid F\in A\right>_{\bbk}$.

\subsection{Decompositions and compatibility}
\label{subs:contr-compat}
Let  $\q=\f \oplus V$ be a vector space decomposition,
where $\f$ is a subalgebra. For any $s\in\bbk^{\times}$, define a linear map
$\vp_s\!:\q\to\q$ by
setting $\vp_s\vert_{\f}=\id$, $\vp_s\vert_{V}=s{\cdot}\id$. Then $\vp_s\vp_{s'}=\vp_{ss'}$ and
$\vp_s^{-1}=\vp_{s^{-1}}$, i.e., this yields a one-parameter subgroup of $\mathrm{GL}(\q)$.
For each $s$, the formula
\beq     \label{eq:fi_s}
    [x,y]_{(s)}=\vp_s^{-1}([\vp_s(x),\vp_s(y)])
\eeq
defines a modified Lie algebra structure 
on the vector space $\q$. 
All these structures are isomorphic to the initial one. The corresponding Poisson bracket is denoted by
$\{\,\,,\,\}_{(s)}$. We naturally extend $\vp_s$ to an automorphism of $\gS(\q)$.  Then the centre of the
Poisson algebra $(\gS(\q),\{\,\,,\,\}_{(s)})$ equals $\vp_s^{-1}(\gS(\q)^{\q})$.  For $x\in\q$, write
$x=x_{\f}+x_V$ with $x_{\f}\in\f$, $x_V\in V$.

If $\q=\f \oplus V$ is a $\BZ_2$-grading,
i.e., $[\f,V]\subset V$ and $[V,V]\subset\f$, then $\{\,\,,\,\}_{(s)}=\{\,\,,\,\}_{(-s)}$ and
$\{\,\,,\,\}_{(s)}+\{\,\,,\,\}_{(s')}=2\{\,\,,\,\}_{(\tilde s)}$ with $2\tilde s^2=s^2+(s')^2$.
Therefore the  brackets $\{\,\,,\,\}_{(s)}$ are pairwise compatible and build a two-dimensional pencil.

\begin{lm}           \label{lm:dec}
Suppose that $\q=\f\oplus V$, where $\f\subset\q$ is a subalgebra and
$[\f,V]\subset V$. For any $x=x_\f+x_v, y=y_\f+y_V \in \q$, we have
\begin{equation} \label{dec-s}
  [x,y]_{(s)}=[x_{\f},x_\f] + [x_{\f},y_V]+[x_V,y_{\f}] + s[x_V,y_V]_V +s^2[x_V,y_V]_{\f} .
\end{equation}
\end{lm}
\begin{proof}
The statement is verified by a straightforward computation.
\end{proof}

Assume that  $\q=\f\oplus V$ is an $\f$-stable decomposition. One of the  crucial properties of
$[\,\,,\,]_{(s)}$ is that if $x\in\f$ and $y\in\q$, then $[x,y]_{(s)}=[x,y]$ for all  $s\in\bbk$.
Then~\eqref{dec-s} shows also that  if $[\f,\q]\ne 0$ and  $[V,V]$ is not contained in either
$\f$ or $V$, then the  brackets $\{\,\,,\,\}_{(s)}$ do not build a two-dimensional pencil.

\section{A criterion for commutativity}
\label{sect:2}

\noindent
Let $\f$ be a subalgebra of $\q$. Suppose that there is an $\f$-stable decomposition $\q=\f\oplus\me$,
i.e., $[\f,\me]\subset\me$.  This yields a bi-homogeneous structure for $\gS(\q)$:
\[
    \gS(\q)=\bigoplus_{i,j \ge 0} \gS^i(\f)\otimes \gS^j(\me) .
\]
For any $H\in\gS(\q)$, we have $H=\sum_{i,j\ge 0}H_{(i,j)}$, where
$H_{(i,j)}\in \gS^i(\f)\otimes \gS^j(\me)$.
Let $\gZ_{(\q,\f)}$ be the subalgebra of $\gS(\g)$ generated by the bi-homogeneous components of
all $H\in\gS(\q)^{\q}$. Since each bi-homogeneous component of $H\in\gS(\q)^{\q}$ is
$\f$-invariant, we have $\gZ_{(\q,\f)}\subset\gS(\q)^{\f}$. It is claimed in~\cite[Theorem\,1]{bo} that if $\q$
is semisimple and $\f$ is reductive (so that an $\f$-stable decomposition of $\q$ does exist), then
$\gZ_{(\q,\f)}$ is Poisson commutative. However, this is {\bf false}! Below, we give a criterion for
the Poisson commutativity of $\gZ_{(\q,\f)}$ and provide a counterexample to the assertion of~\cite{bo}.
On the positive side, we deduce from our criterion that $\gZ_{(\q,\f)}$ is
Poisson commutative whenever $\f$ is an abelian subalgebra.

Given $\gamma\in\q^*$, we decompose it as $\gamma=\gamma_{\f}+\gamma_{\me}$, where
$\gamma_{\f}\vert_\me=0$ and $\gamma_{\me}\vert_\f=0$. Let $\vp_s : \q\to \q$ be the same as in
Section~\ref{subs:contr-compat} with $V=\me$. Set $\vp_s(\gamma)=\gamma_{\f}+s\gamma_{\me}$. It is well known and easily verified that, for any
$H\in\gS(\q)^{\q}$ and $\xi\in\q^*$, one has $\textsl{d}_\xi H\in \z(\q^\xi)$, where
$\z(\q^\xi)$ is the centre of $\q^\xi$. A standard calculation with differentials shows that
\beq      \label{d-phi}
        \textsl{d}_\gamma (\vp_s(F))=\vp_s(\textsl{d}_{\vp_s(\gamma)}F)
\eeq
for any $F\in\gS(\q)$.

\begin{thm}        \label{thm:com}
The subalgebra $\gZ=\gZ_{(\q,\f)}$ is Poisson commutative if and only if
\[
     \hat\gamma_{\f} \bigl((\textsl{d}_{\vp_s(\gamma)}H)_{\f},(\textsl{d}_{\vp_{s'}(\gamma)}H')_{\f}\bigr)=0
\]
for each  $\gamma\in\q^*$, all nonzero $s,s'\in\bbk$, and all $H,H'\in\gS(\q)^{\q}$.
\end{thm}
\begin{proof}
It suffices to prove the assertion for homogeneous $H, H'\in\gS(\q)^\q$. Note that if $H\in\gS^d(\q)$ and
$H=\sum_{j=0}^d H_{(d-j,j)}$, then $\vp_s(H)=\sum_j s^j H_{(d-j,j)}$. Therefore,
employing 
the standard argument with the Vandermonde determinant, one shows that
\beq           \label{Vand}
          \gZ=\mathsf{alg}\langle\vp_s(H)\mid H\in\gS(\q)^{\q},s\in \bbk^\times\rangle.
\eeq
Hence the algebra $\gZ$ is Poisson commutative if and only if for all $H,H'\in\gS(\q)^{\q}$, all nonzero
$s,s'$, and any  $\gamma\in\q^*$, we have
\[
   A_{s,s'}=A_{s,s',H,H',\gamma}:=\hat\gamma(\textsl{d}_\gamma \vp_s(H),\textsl{d}_\gamma\vp_{s'}(H'))=0.
\]
Suppose that $H$, $H'$ and  $\gamma$ are fixed. Then there is no ambiguity in the use of $A_{s,s'}$.

Set $\xi=\textsl{d}_{\vp_s(\gamma)}H$ and $\eta=\textsl{d}_{\vp_{s'}(\gamma)}H'$. Since $\vp_s(H)$ belongs
to the Poisson centre of $(\gS(\q),\{\,\,,\,\}_{(s^{-1})})$, we derive from~\eqref{d-phi} that
\[
      \gamma([\textsl{d}_\gamma \vp_s(H),\textsl{d}_\gamma \vp_{s'}(H')]_{(s^{-1})})= \gamma([\vp_s(\xi),\vp_{s'}(\eta)]_{(s^{-1})})=\gamma([\xi_{\f}+s\xi_{\me},\eta_{\f}+s'\eta_{\me}]_{(s^{-1})})=0.
\]
Similarly, $\vp_{s'}(H')$ belongs to the Poisson centre of $(\gS(\q),\{\,\,,\,\}_{((s')^{-1})})$ and hence
\[
\gamma([\xi_{\f}+s\xi_{\me},\eta_{\f}+s'\eta_{\me}]_{((s')^{-1})})=0.
\]
For all $\tilde s\in\bbk^\times$ and $\tilde H\in\gS(\q)^{\q}$, we have
$\hat\gamma(\f,\textsl{d}_\gamma\vp_{\tilde s}(\tilde H))=0$.  Therefore,
$\hat\gamma(\vp_s(\xi),\eta_{\f})=\hat\gamma(\xi_{\f},\vp_{s'}(\eta))=0$. Thus,
\beq       \label{f}
     C:=\hat\gamma_{\f}((\textsl{d}_{\vp_s(\gamma)}H)_{\f},(\textsl{d}_{\vp_{s'}(\gamma)}H')_{\f})=
     \hat\gamma(\xi_{\f},\eta_{\f})=-s'\hat\gamma(\xi_{\f},\eta_{\me})=-s \hat\gamma(\xi_{\me},\eta_{\f}).
\eeq
Let us substitute this into the formulas
\begin{align*}
& \gamma([\vp_s(\xi),\vp_{s'}(\eta)]_{(s^{-1})})= \\
&\qquad \hat\gamma(\xi_{\f},\eta_\f)+s\hat\gamma(\xi_{\me},\eta_\f)+s'\hat\gamma(\xi_{\f},\eta_{\me})+\frac{s'}{s}\gamma_{\f}([\xi_{\me},\eta_{\me}]_{\f}) +s'\gamma_{\me}([\xi_{\me},\eta_{\me}]_{\me}) =0, \\
& \gamma([\vp_s(\xi),\vp_{s'}(\eta)]_{((s')^{-1})})= \\
& \qquad \hat\gamma(\xi_{\f},\eta_\f)+s\hat\gamma(\xi_{\me},\eta_\f)+s'\hat\gamma(\xi_{\f},\eta_{\me})+\frac{s}{s'}\gamma_{\f}([\xi_{\me},\eta_{\me}]_{\f}) +s\gamma_{\me}([\xi_{\me},\eta_{\me}]_{\me})=0,
\end{align*}
obtaining the equalities
\begin{align*}
& C-C-C + s^{-1}s'\gamma_{\f}([\xi_{\me},\eta_{\me}]_{\f}) +s'\gamma_{\me}([\xi_{\me},\eta_{\me}]_{\me}) = 0, \\
& C-C -C + s(s')^{-1}\gamma_{\f}([\xi_{\me},\eta_{\me}]_{\f}) +s\gamma_{\me}([\xi_{\me},\eta_{\me}]_{\me}) =0.
\end{align*}
Furthermore
$A_{s,s'}=-C+ ss'\gamma_{\f}([\xi_{\me},\eta_{\me}]_{\f}) +ss'\gamma_{\me}([\xi_{\me},\eta_{\me}]_{\me})$.

Suppose  that $C=0$, then
\begin{gather*}
s^{-1}\gamma_{\f}([\xi_{\me},\eta_{\me}]_{\f}) +\gamma_{\me}([\xi_{\me},\eta_{\me}]_{\me}) = 0, \\
(s')^{-1}\gamma_{\f}([\xi_{\me},\eta_{\me}]_{\f}) +\gamma_{\me}([\xi_{\me},\eta_{\me}]_{\me}) =0.
\end{gather*}
Thereby $(s^{-1}-(s')^{-1})\gamma_{\f}([\xi_{\me},\eta_{\me}]_{\f})=0$ and
$(s-s')\gamma_{\me}([\xi_{\me},\eta_{\me}]_{\me})=0$. If $s\ne s'$, then necessary  $A_{s,s'}=0$.
Since $A_{s,s'}$ is a polynomial in $s$ and $s'$ with constant coefficients, $A_{s,s'}=0$ for all nonzero
$s,s'$. This settles the `if' part.

In order to prove the `only if' implication, suppose that $A_{s,s'}=0$ for all  $s,s'\in\bbk^\times$. Then
$x=\gamma_{\f}([\xi_{\me},\eta_{\me}]_{\f})$ and
$y=\gamma_{\me}([\xi_{\me},\eta_{\me}]_{\me})$ satisfy
$s^{-1}s'x+s'y= s (s')^{-1} x +sy = ss'(x+y)=C$.
Assume that $s\ne s'$ and that $s,s'\ne 1$. Then
\[
\left\{ \begin{array}{r}
\frac{s'+s}{ss'}{\cdot} x + y = 0; \\
\frac{s+1}{s}{\cdot} x +  y=0,
    \end{array}
\right.
\]
and  the only solution of this system  is $x=y=0$. Hence $C=0$.
By the continuity in $s$ and $s'$, the equality $\hat\gamma_{\f}((\textsl{d}_{\vp_s(\gamma)}H)_{\f},(\textsl{d}_{\vp_{s'}(\gamma)}H')_{\f})=0$ holds for all  $s,s'\in\bbk^\times$.
\end{proof}

\begin{cl}          \label{cor:f-abelian}
If\/ $\f$ is an abelian Lie algebra, then $\gZ_{(\q,\f)}$ is Poisson commutative.
\end{cl}
\begin{proof}
Since $[\f,\f]=0$, we have
$[(\textsl{d}_{\vp_s(\gamma)}H)_{\f},(\textsl{d}_{\vp_{s'}(\gamma)}H')_{\f}]=0$
for each  $\gamma\in\q^*$, all nonzero $s,s'\in\bbk$, and all $H,H'\in\gS(\q)^{\q}$.
Hence $\gZ_{(\q,\f)}$ is Poisson-commutative by Theorem~\ref{thm:com}.
\end{proof}

Let $\g$ be a reductive Lie algebra. Then $\g$ is identified with $\g^*$ via a $G$-invariant non-degenerate
scalar product $(\,\,,\,)$ and $\gS(\g)^\g$ is a polynomial ring.
Let $\{H_1,\dots,H_l\}$ be a set of homogeneous algebraically independent generators of $\gS(\g)^\g$
with $\deg H_j=:d_j$.
By the {\it Kostant regularity criterion\/} for $\g$~\cite[Theorem~9]{ko63},
\beq          \label{eq:ko-re-cr}
      \text{ $\langle\textsl{d}_\xi H_j \mid 1\le j\le l\rangle_{\bbk}=\g^\xi$ \ if and only if \ $\xi\in\g^*_{\sf reg}$.}
\eeq
Recall that $\g^\xi=\z(\g^\xi)$ if and only if $\xi\in\g^*_{\sf reg}$~\cite[Theorem\,3.3]{p03}.

\begin{ex}       \label{contra-4}
If $\g=\gln$, then $x^k\in\g^x$ for any $x\in\g$ and $k\in\BN$. (Here $x^k$ is the usual matrix power.)
Moreover, if we identify $\g$ and $\g^*$, then
$\textsl{d}_x \gS(\g)^{\g}=\langle x^k \mid 0\le k < n\rangle_{\bbk}$.

Consider the pair $(\g,\f)=(\gt{gl}_4,\tri)$ with $\tri$ embedded in the right lower corner.
\\[.7ex]
\centerline{Take
$\gamma=\begin{pmatrix}
                       1 & 0 & 0 & 1 \\
                       0 & 0 & 1 & 0 \\
                       1 & 0 & 1 & 0 \\
                       0 & 1 & 0 & -1 \\
                       \end{pmatrix}$.  \
Then $\gamma_\f+s\gamma_{\me}=
\begin{pmatrix}
                       s & 0 & 0 & s \\
                       0 & 0 & s & 0 \\
                       s & 0 & 1 & 0 \\
                       0 & s & 0 & -1 \\
                       \end{pmatrix}$.  }
\\[.6ex]
Note that $\gamma_{\f}\ne 0$.
For any $k\ge 0$, $(\vp_s(\gamma))^k=(\gamma_{\f}+s\gamma_{\me})^k$ belongs to
$\textsl{d}_{\vp_s(\gamma)} \gS(\g)^{\g}$. Hence
$((\gamma_{\f}+s\gamma_{\me})^k)_{\f} \in (\textsl{d}_{\vp_s(\gamma)}\gS(\g)^{\g})_{\f}$.
Let us do calculations for $k=2,3$:
\[
(\gamma_\f+s\gamma_{\me})^2=\begin{pmatrix} 
s^2      & s^2 & 0 & s^2 {-} s\\
s^2      & 0 & s & 0 \\
s^2 {+} s & 0 & 1 & s ^2 \\
0          & -s & s^2 & 1 \\
\end{pmatrix}  
\quad \text{ and } \quad
(\gamma_\f+s\gamma_{\me})^3=\left(\begin{array}{cccc}
*  & *  & *  &  * \\
* & * & * & *  \\
* & * & 1 & s^3 \\
* & * & 0 & -1 \\
\end{array}\right).
\]
Let $e=\big(\begin{smallmatrix} 0 & 1\\ 0& 0\end{smallmatrix}\big)$,
$h=\big(\begin{smallmatrix} 1 & 0\\ 0& -1\end{smallmatrix}\big)$, and
$f=\big(\begin{smallmatrix} 0 & 0\\ 1& 0\end{smallmatrix}\big)$ be the standard basis of $\tri$. Then
\\[.6ex]
\centerline{$\gamma_{\f}=h$, \ $((\gamma_\f+s\gamma_{\me})^2)_{\f} = s^2(e+f)$, \ and
$((\gamma_\f+s\gamma_{\me})^3)_{\f} = s^3 e+ h$.}
\\[.6ex]
Therefore, if $s\ne 0$, then
$\left<(\textsl{d}_{\vp_s(\gamma)} H)_{\f}\mid H\in\gS(\g)^{\g}\right>_{\bbk}=\f$.
Since $(h,[\f,\f])\ne 0$, we conclude that
\[
    \hat\gamma_{\f}((\textsl{d}_{\vp_s(\gamma)}H)_{\f},(\textsl{d}_{\vp_{s'}(\gamma)}H')_{\f})\ne 0
\]
for all nonzero $s,s'$. Thus, by Theorem~\ref{thm:com}, $\gZ_{(\g,\f)}$ is not Poisson commutative.
\end{ex}

\begin{rmk} \label{rem-sl}
Example~\ref{contra-4} also implies that $\gZ_{(\g,\tri)}$ is not Poisson commutative if
$\g=\mathfrak{gl}_4$ is replaced with $\mathfrak{sl}_4$.
For, $\mathfrak{gl}_4=\z \oplus \mathfrak{sl}_4$ with $\z=\bbk I_4$, hence
$\gS(\mathfrak{gl}_4)^{\mathfrak{gl}_4}$ is generated by $\gS(\mathfrak{sl}_4)^{\gt{sl}_4}$ and $\z$.
For any reductive $\f\subset\mathfrak{sl}_4$, the algebra
$\gZ_{(\mathfrak{gl}_4,\f)}= \mathsf{alg}\langle \gZ_{(\gt{sl}_4,\f)},\z\rangle$
is Poisson commutative if and only if  $\gZ_{(\mathfrak{sl}_4,\f)}$ is.
\end{rmk}

Example~\ref{contra-4} easily generalises to the pairs $(\gln,\gt{gl}_m)$ with $n\ge m+2$. On the other
hand, one can prove that the algebra $\gZ_{(\gt{gl}_3,\tri)}$ or $\gZ_{(\gt{sl}_3,\tri)}$ is still
Poisson commutative.

\begin{ex}       \label{ex-MF}
Let us show that, for a special choice of $\f$,  the algebra $\gZ_{(\g,\f)}$ is rather close to
an  ${\eus MF}$-subalgebra.
\\  \indent
Let $h\in\g$ be a semisimple element such that $(h,h)\ne 0$. Set  $\f=\lg h\rg$. Then $\me\subset\g$ is
the orthogonal complement of $h$ with respect to $(\,\,,\,)$ and 
the bi-homogeneous decomposition of $H_j\in\gS(\g)^\g$ is
\[
      H_j=H_{j,0} h^{d_j} + H_{j,1}h^{d_j-1}+\ldots+ H_{j,k}h^{d_j-k} +\ldots+ H_{j,d_j},
\]
where $H_{j,k}\in\gS^k(\gt m)$. By definition, $\gZ_{(\g,\lg h\rg)}$ is generated by $H_{j,k}h^{d_j-k}$ with
$1\le j\le l$ and $0\le k\le d_j$. On the one hand, we had $\f=\lg h\rg$.
 On the other hand,
let $\gamma\in\g^*$ be such that $\gamma(\gt m)=0$ and
$\gamma(h)=1$. Actually, $\gamma=h$ under the identification of $\g$ and $\g^*$.
Then 
\[
     \partial_\gamma^k H_j=\sum_{r=k}^{d_j} r(r-1)\ldots(r-k+1)h^{r-k}H_{j,d_j-r} .
\]
If $H\in\gS^2(\g)^\g$ is the quadratic form corresponding to $(\,\,,\,)$, then $\partial_\gamma H=ch$ for
some $c\in\bbk^\times$. Hence $h\in (\cam)_\gamma$.
Arguing by induction on $k$, we obtain $H_{j,k}\in (\cam)_\gamma$ for $k\le d_j$. Thus
\[
    \gZ_{(\g,\lg h\rg)} \subset (\cam)_\gamma =(\cam)_h \subset \mathsf{alg}\left<\gZ_{(\g,\lg h\rg)}, h,h^{-1}\right>.
\]
\end{ex}

\section{Properties of the algebra  $\gZ_{(\g,\te)}$}
\label{sect:3}

\noindent
Suppose that $\g$ is semisimple. Let $\te$ be a Cartan subalgebra of $\g$ and $\Delta$ the root system of $(\g,\te)$. By Corollary~\ref{cor:f-abelian}, the algebra $\gZ_{(\g,\te)}$ is Poisson commutative, and our goal is to
prove that this algebra has a number of remarkable properties.
Let $\g_\gamma$ be the root space corresponding to $\gamma\in\Delta$ and let $e_\gamma\in
\g_\gamma$ be a nonzero vector.
Then $\me=\te^\perp=\bigoplus_{\gamma\in\Delta}\g_\gamma$.

Recall that $\{H_1,\dots,H_l\}$ is a set of homogeneous algebraically independent generators of
$\gS(\g)^\g$ and $\deg H_j=d_j$. One has $\sum_{j=1}^l d_j=\bb(\g)$. The vector space decomposition
$\g=\te\oplus\me$ provides the bi-homogeneous decomposition of each $H_j$:
\[
    H_j=\sum_{i=0}^{\textsl{d}_j} (H_j)_{(i,\textsl{d}_j-i)} ,
\]
where $(H_j)_{(i,\textsl{d}_j-i)}\in \gS^i(\te)\otimes \gS^{\textsl{d}_j-i}(\me)\subset \gS^{d_j}(\g)$.
Recall that $\gZ:=\gZ_{(\g,\te)}$ is the algebra generated by
\beq     \label{eq:bihom}
    \{ (H_j)_{(i,d_j-i)} \mid j=1,\dots,l; i=0,1,\dots, d_j\} .
\eeq
Since each $H_j$ is $\g$-invariant, all the bi-homogeneous components in~\eqref{eq:bihom} are
$\te$-invariant. Hence $\gZ\subset\gS(\g)^{\te}$. The total number of these functions is
$\sum_{j=1}^l (d_j+1)=\bb(\g)+l$, but some of them are identically equal to zero.
Indeed, $(H_j)_{(d_j-1,1)}\in \gS^{d_j-1}(\te)\otimes \me$ and
$\me^\te=\{0\}$, hence $(H_j)_{(d_j-1,1)} \equiv 0$ for $j=1,\dots,l$. Therefore, the number of nonzero
generators of $\gZ$ is at most $\bb(\g)$.

The bi-homogeneous component $(H_j)_{(d_j,0)}\in \gS^{d_j}(\te)$ is the restriction of $H_j$ to
$\te\simeq\te^*$. Therefore, by the Chevalley restriction theorem, the polynomials $(H_j)_{(d_j,0)}$,
$j=1,\dots,l$, are the free generators of $\gS(\te)^W$, where $W$ is the Weyl group of $\te$. This
means that having replaced $(H_1)_{(d_1,0)},\dots, (H_l)_{(d_l,0)}$ with a basis of $\te$ and keeping
intact all other bi-homogeneous components (generators of $\gZ$), we obtain a larger subalgebra
$\tilde\gZ$, which is an algebraic extension of $\gZ$. Furthermore, since $\gZ\subset\gS(\g)^\te$,
$\tilde\gZ$  is still Poisson commutative.

Once again, we use the map $\vp_s$ defined in Section~\ref{subs:contr-compat}.
By~\eqref{eq:ko-re-cr}, if $\vp_s(\gamma)\in\g^*_{\sf reg}$, then $\textsl{d}_{\vp_s(\gamma)} \gS(\g)^{\g}=\g^{\vp_s(\gamma)}$ ; and by~\eqref{d-phi}, we have
\begin{equation} \label{d-phi2}
\textsl{d}_\gamma \vp_s(\gS(\g)^{\g})=\vp_s(\g^{\vp_s(\gamma)}).
\end{equation}
As before, we identify $\g$ and $\g^*$.

\begin{lm}\label{rel-MF}
Let $h\in\te$ and $x\in\me$ be such that $(h + \bbk x)\cap \g^*_{\sf reg}\ne\varnothing$. Then
$\textsl{d}_{h+x}\tilde\gZ =\te +\textsl{d}_h((\cam)_x)$. Moreover, if $h\in\g^*_{\sf reg}$, then
$\textsl{d}_{h+x}\tilde\gZ = \textsl{d}_h((\cam)_x)=\textsl{d}_x((\cam)_h)$.
\end{lm}
\begin{proof}
The assumption  $(h\oplus\bbk x)\cap \g^*_{\sf reg}\ne\varnothing$ implies that
$\Omega:=\{s\in \bbk^\times\mid  h+sx\in\g^*_{\sf reg}\}$ is a nonempty open subset of $\bbk^\times$. Since $\Omega$ is infinite, we can strengthen~\eqref{Vand} as
\beq          \label{Vand-O}
             \gZ=\mathsf{alg}\langle \vp_s(H)\mid H\in\gS(\g)^{\g},s\in \Omega\rangle.
\eeq
Combining this with~\eqref{d-phi2}, we obtain 
\[
   \textsl{d}_{h+x}\tilde\gZ =\te +\sum_{s\in\Omega} \textsl{d}_{h+x}\vp_s(\gS(\g)^{\g})= \te+
   \sum_{s\in\Omega} \vp_s(\g^{h+sx}) =\te +\sum_{s\in\Omega} \g^{h+sx}.
\]
Set $\Omega'=\Omega\sqcup\{0\}$ if $h\in\g^*_{\sf reg}$ and  $\Omega'=\Omega$ otherwise. Then
the equality $\sum_{s\in\Omega'} \g^{h+sx}= \textsl{d}_h((\cam)_x)$ follows from
\cite[Lemma~1.3]{kruks}, see also the proof of Lemma~2.1 in~\cite[Sect.~2]{kruks}.
If $h\not\in\g^*_{\sf reg}$, we are done.
If $h\in\g^*_{\sf reg}$, then $\g^h=\te$ and $\te +\sum_{s\in\Omega} \g^{h+sx}=\sum_{s\in\Omega'} \g^{h+sx}$.

Finally, we recall that $\textsl{d}_h((\cam)_x)=\textsl{d}_x((\cam)_h)$ for any $x,h\in\g$ by \cite[Eq.~(2${\cdot}$3)]{kruks}. 
\end{proof}

\begin{thm}         \label{thm:free-max}
For $\gZ=\gZ_{(\g,\te)}$ and $\tilde\gZ$ as above, we have
\begin{itemize}
\item[\sf (i)] \  $\trdeg\gZ=\trdeg\tilde\gZ=\bb(\g)$ and both algebras $\gZ$ and $\tilde\gZ$ are polynomial;
\item[\sf (ii)] \ both $\gZ$ and $\tilde\gZ$ are complete on each regular orbit;
\item[\sf (iii)] \ $\tilde\gZ$ is a maximal Poisson commutative subalgebra of\/ $\gS(\g)$.
\end{itemize}
\end{thm}
\begin{proof}
{\sf (i)} \ Since $\gZ\subset\tilde\gZ$ is an algebraic extension, the first equality follows.
Take a principal
$\tri$-triple $\{e,h,f\}\subset\g$ such that $h\in\te$ and $e,f \in\me$. Note that any nonzero
element of $\lg e,h,f\rg$ is regular in $\g\simeq\g^*$. Pick a nonzero $x\in\langle e,f \rangle
\subset\me$ and consider the subspace
\[
     \textsl{d}_{h+x}\tilde\gZ:=\{ \textsl{d}_{h+x} F\mid  F\in\tilde\gZ\} \subset \g.
\]
Because $\vp_s(h+x)=h+sx\in\g^*_{\sf reg}$ for all $s$, we have
$\textsl{d}_{h+x}\tilde\gZ =\textsl{d}_h((\cam)_x)$ by Lemma~\ref{rel-MF}.

Using properties of $\eus{MF}$-subalgebras, see~\cite{mrl},~\cite[Cor.~1.6\,\&\,Lemma~2.1]{kruks}
and the fact that $(\bbk x\oplus\bbk h)\cap\g_{\sf sing}=\{0\}$, we obtain
$\dim\textsl{d}_{h+x}\tilde\gZ=\bb(\g)$.
It follows that $\trdeg\tilde\gZ\ge \bb(\g)$, and
since $\trdeg\ca\le \bb(\g)$ for any \PC\ subalgebra, we actually get the equality.
As both
$\gZ$ and $\tilde\gZ$ have at most $\bb(\g)$ generators, they are polynomial.
Therefore, $\gZ$ is freely generated by
\[
   \{(H_j)_{(i,d_j-i)}\mid 1\le j\le l;\ i=0,1,\dots,d_j-2,d_j\} ,
\]
while $\tilde\gZ$ is freely generated by a basis of $\te$ and the components $(H_j)_{(i,d_j-i)}$, where    $1\le j\le l$ and $i=0,1,\dots, d_j-2$.

{\sf (ii)} \ In part {\sf (i)}, we proved that $\dim \textsl{d}_{h+x}\tilde\gZ=\bb(g)$.
Then~\cite[Lemma\,1.2]{kruks} implies that $\tilde\gZ$ is complete on the orbit $G(h+x)$. For an appropriate
choice of $x\in\langle e,f\rangle$, we obtain a nilpotent element $h+x\in\lg e,h,f\rg\simeq\tri$. Hence $\tilde\gZ$ is complete on the regular nilpotent orbit. Then a standard deformation argument,
see~\cite[Cor.\,2.6]{kruks}, shows that $\tilde\gZ$ is complete on every regular orbit.
The same line of argument applies to $\gZ$, since $\textsl{d}_{h+x}\gS(\te)^W=\textsl{d}_h\gS(\g)^\g=\te$ and
$\textsl{d}_{h+x}\gZ=\textsl{d}_{h+x}\tilde\gZ$.

{\sf (iii)} \ The maximality of $\tilde\gZ$ will follow from the fact that
$Y=\{\gamma\in\g^*\mid\dim \textsl{d}_\gamma \tilde\gZ<\bb(\g)\}$ is of codimension ${\ge}2$ in $\g^*$
(see below). We identify $\g$ and $\g^*$ via the Killing form and regard $Y$ as a subvariety of
$\g$. Write $\gamma=h'+x'$ with $h'\in\te$,
$x'\in\me$. If $\lg h',x'\rg \cap \g_{\sf sing} = \{0\}$, then
$\dim\textsl{d}_h((\cam)_x)=\bb(\g)$~\cite[Theorem\,2.5]{mrl} and
$\dim\textsl{d}_\gamma\tilde\gZ=\bb(\g)$ by Lemma~\ref{rel-MF}.

Consider the map $\psi: \g_{\sf sing} \times \bbk \to \g$ defined by
$\psi(\xi,s)=\xi_\te+s\xi_{\me}$ and let $\tilde Y$ be the closure of $\Ima(\psi)$.
Set $\te_{\sf sing}:=\te\cap\g_{\sf sing}$ and $\me_{\sf sing}:=\me\cap\g_{\sf sing}$. Then
\[
  Y\subset \tilde Y\cup (\te_{\sf sing}\times \me) \cup (\te\times \me_{\sf sing}) .
\]
\indent
\textbullet \quad Since $\codim\g_{\sf sing}=3$, we have $\dim\tilde Y\le \dim\g-2$.
\\ \indent
\textbullet \quad As $\me_{\sf sing}$ is conical and $\lg e,f\rg \cap \me_{\sf sing}=\{0\}$, we have
$\dim\me_{\sf sing}\le \dim\me-2$. Therefore, $\te\times \me_{\sf sing}\subset\g$ does not contain divisors.
\\ \indent
\textbullet \quad We prove below that $\dim(Y\cap (\te_{\sf sing}\times \me))\le\dim\g -2$, which
yields the required estimate of $\codim Y$.

The subset $\te_{\sf sing}\subset\te$ is the union of all reflection hyperplanes in $\te$. That is,
if 
\[\eus H_\gamma=\{x\in\te\mid (\gamma,x)=0\},
\] then
$\te_{\sf sing}=\bigcup_{\gamma\in\Delta}\eus H_\gamma$. (Of course,
$\eus H_\gamma=\eus H_{-\gamma}$.) Suppose that $h'\in \eus H_\nu$ is generic, i.e.,
$h'\in \eus H_\nu\setminus \bigcup_{\gamma\ne\pm\nu}\eus H_\gamma$.
Then $h'\in\g$ is subregular and
\[
   \g^{h'}=\te\oplus\g_\nu\oplus\g_{-\nu}=\eus H_\nu\oplus\lg e_\nu, h_\nu, e_{-\nu}\rg
   \simeq \eus H_\nu\oplus\tri ,
\]
where $h_\nu=[e_\nu,e_{-\nu}]$.
Note also that $\eus H_\nu=\textsl{d}_{h'}\gS(\g)^{\g}\subset\te$, cf.~\cite[Lemma~4.9]{OY-alg}.
Without loss of generality, we may assume that $\nu$ is a {\bf simple} root with respect to some choice
of $\Delta^+\subset\Delta$. Let $\Pi\subset\Delta^+$ be the corresponding set of simple roots and
$\me=\ut^+\oplus\ut^-$. We may also assume that
$e=\sum_{\ap\in\Pi}c_\ap e_\ap\in\ut$ with $c_\ap\in\bbk^\times$ and $f=\sum_{\ap\in\Pi}e_{-\ap}\in\ut^-$ for a principal $\tri$-triple
$\{e,h,f\}$ with $h\in\te$, cf.~\cite[Theorem\,4]{ko63}. Then $f+\be\subset\g_{\sf reg}$
by~\cite[Lemma\,10]{ko63}. In particular, $h'+sf\in\g_{\sf reg}$ for any $s\in\bbk^\times$.

\begin{lm}          \label{lm:cent-subreg}
If  $h'\in\eus H_\nu$ is generic, then $\g^{h'+sf}\subset \eus H_\nu\oplus\ut^-$ for any $s\ne 0$.
\end{lm}
\begin{proof}
As is well known,  $(\g^{h'+sf})^\perp=[\g, h'+sf]$. Hence it suffices to prove that
$[\g, h'+sf]\supset (\eus H_\nu\oplus\ut^-)^\perp=\lg h_\nu\rg \oplus\ut^-$.

Since $h'+sf\in\be^-$ is regular in $\g$, we have $\g^{h'+sf}\subset \be^-$. Hence
$[\g, h'+sf]\supset (\be^-)^\perp=\ut^-$. Next,
$[e_\nu, h'+sf]=[e_\nu,sf]=s[e_\nu,e_{-\nu}]=s{\cdot}h_\nu\in [\g,h'+sf]$.
\end{proof}
Now, set
\[
     \BV:= \textsl{d}_{h'} ((\cam)_{f})=\sum_{s\ne 0} \g^{h'+sf},
\]
where the last equality stems from~\cite[Lemma~1.3]{kruks}. 
On the one hand,
$\BV\subset\eus H_\nu\oplus\ut^-$ by the above lemma. On the other hand, $\dim\BV=\bb(\g)-1$ in view
of~\cite[proof of Theorem~2.4]{kruks}. Hence $\BV=\eus H_\nu\oplus\ut^-$.

The differentials $\textsl{d}_{h'+f}((H_j)_{(d_j,0)})=\textsl{d}_{h'}((H_j)_{(d_j,0)})=\textsl{d}_{h'}H_j$ with $1\le j\le l$ are
linearly dependent, hence 
$\dim\textsl{d}_{h'+f} \gZ\le \bb(\g)-1$. 
Recall that $h'+sf\in\g_{\sf reg}$ for any $s\in\bbk^\times$. Combining~\eqref{Vand-O} with \eqref{d-phi2}, we obtain  
\[
   \BV_{\gZ}:=\textsl{d}_{h'+f} \gZ\supset \sum_{s\ne 0} \vp_s(\g^{h'+sf})=\vp_s(\BV)=\eus H_\nu\oplus\ut^-.
\]
Thus 
$\mathbb V_{\gZ}=\eus H_\nu+\ut^-$. Next $\textsl{d}_{h'+f}\tilde\gZ \ne \textsl{d}_{h'+f}\gZ$, since $\te\not\subset\BV_{\gZ}$.
We obtain $\dim\textsl{d}_{h'+f}\tilde\gZ=\bb(\g)$, which means that
$\dim\textsl{d}_{\gamma}\tilde\gZ=\bb(\g)$ on a dense open subset of $\eus H_\nu\times\me$.
Since $\nu\in\Delta$ is arbitrary, this implies that
$\dim (Y\cap (\te_{\sf sing}\times\me))\le \dim\g-2$.
\\ \indent
Thus, we have proved that $\dim Y\le \dim\g-2$.

Since $\tilde\gZ$ is generated by algebraically independent homogeneous polynomials and
$\codim Y\ge 2$, it follows from~\cite[Theorem~1.1]{ppy} that $\tilde\gZ$ is an algebraically closed subalgebra of $\gS(\g)$ (i.e., if $F\in \gS(\g)$ is algebraic over the quotient field of $\tilde\gZ$, then $F\in \tilde\gZ$).
An inclusion $\tilde\gZ\subset\gA\subset\gS(\g)$, where $\{\gA,\gA\}=0$, is
only possible if $\gA$ is an algebraic extension of $\tilde\gZ$, because $\trdeg\tilde\gZ=\bb(\g)$ and
$\trdeg\gA\le \bb(\g)$. Therefore we must have $\tilde\gZ=\gA$.
\end{proof}

\begin{rmk}
We know that $\tilde\gZ\subset \gS(\g)^\te$ and $\trdeg\tilde\gZ=\bb(\g)$. If $h\in\te^*_{\sf reg}$, then
these two properties are also satisfied for $(\cam)_{h}$~\cite{mrl}.
One may say that $\tilde\gZ$ resembles all such $\eus{MF}$-subalgebras.
However, there is no choice of $h\in\te^*$ involved in the construction of $\tilde\gZ$ and
$\gZ\subset\gS(\g)^{N_G(\te)}$ unlike any of $(\cam)_{h}$ with $h\in\te^*_{\sf reg}$.
\\ \indent
Furthermore, by Lemma~\ref{rel-MF},
we have 
$\textsl{d}_{h+x}\tilde\gZ=\textsl{d}_h((\cam)_x)=\textsl{d}_x((\cam)_h)$ for any $x\in\me$ and
$h\in\te^*_{\sf reg}$. 
It is tempting to further investigate this relationship.

Another  intriguing task is to produce a quantisation of $\tilde\gZ$, i.e., a commutative
subalgebra of the enveloping algebra $\U(\g)$ such that its associated graded algebra is
$\tilde\gZ$.
\end{rmk}

\end{document}